\documentclass[12pt,reqno]{amsart}
\usepackage{amscd,amssymb,amsmath,multicol,float,scalefnt,cancel,array,stmaryrd,graphicx,tikz}
\restylefloat{table}
\newtheorem{thm}[equation]{Theorem}
\numberwithin{equation}{section}

\newtheorem{lem}[equation]{Lemma}

\newtheorem{fig}[equation]{Figure}

\begin{document}
\raggedbottom \voffset=-.7truein \hoffset=0truein \vsize=8truein
\hsize=6truein \textheight=8truein \textwidth=6truein
\baselineskip=18truept

\def\mapright#1{\ \smash{\mathop{\longrightarrow}\limits^{#1}}\ }
\def\mapleft#1{\smash{\mathop{\longleftarrow}\limits^{#1}}}
\def\mapup#1{\Big\uparrow\rlap{$\vcenter {\hbox {$#1$}}$}}
\def\mapdown#1{\Big\downarrow\rlap{$\vcenter {\hbox {$\ssize{#1}$}}$}}
\def\mapne#1{\nearrow\rlap{$\vcenter {\hbox {$#1$}}$}}
\def\mapse#1{\searrow\rlap{$\vcenter {\hbox {$\ssize{#1}$}}$}}
\def\mapr#1{\smash{\mathop{\rightarrow}\limits^{#1}}}
\def\ss{\smallskip}
\def\vp{v_1^{-1}\pi}
\def\at{{\widetilde\alpha}}
\def\sm{\wedge}
\def\la{\langle}
\def\ra{\rangle}
\def\ev{\text{ev}}
\def\od{\text{od}}
\def\on{\operatorname}
\def\ol#1{\overline{#1}{}}
\def\spin{\on{Spin}}
\def\cat{\on{cat}}
\def\lbar{\ell}
\def\qed{\quad\rule{8pt}{8pt}\bigskip}
\def\ssize{\scriptstyle}
\def\a{\alpha}
\def\bz{{\Bbb Z}}
\def\Rhat{\hat{R}}
\def\im{\on{im}}
\def\ct{\widetilde{C}}
\def\ext{\on{Ext}}
\def\sq{\on{Sq}}
\def\eps{\epsilon}
\def\ar#1{\stackrel {#1}{\rightarrow}}
\def\br{{\bold R}}
\def\bC{{\bold C}}
\def\bA{{\bold A}}
\def\bB{{\bold B}}
\def\bD{{\bold D}}
\def\bh{{\bold H}}
\def\bQ{{\bold Q}}
\def\bP{{\bold P}}
\def\bx{{\bold x}}
\def\bo{{\bold{bo}}}
\def\si{\sigma}
\def\Vbar{{\overline V}}
\def\dbar{{\overline d}}
\def\wbar{{\overline w}}
\def\Sum{\sum}
\def\tfrac{\textstyle\frac}
\def\tb{\textstyle\binom}
\def\Si{\Sigma}
\def\w{\wedge}
\def\equ{\begin{equation}}
\def\b{\beta}
\def\G{\Gamma}
\def\L{\Lambda}
\def\g{\gamma}
\def\k{\kappa}
\def\psit{\widetilde{\Psi}}
\def\tht{\widetilde{\Theta}}
\def\psiu{{\underline{\Psi}}}
\def\thu{{\underline{\Theta}}}
\def\aee{A_{\text{ee}}}
\def\aeo{A_{\text{eo}}}
\def\aoo{A_{\text{oo}}}
\def\aoe{A_{\text{oe}}}
\def\vbar{{\overline v}}
\def\endeq{\end{equation}}
\def\sn{S^{2n+1}}
\def\zp{\bold Z_p}
\def\cR{{\mathcal R}}
\def\P{{\mathcal P}}
\def\cQ{{\mathcal Q}}
\def\cj{{\cal J}}
\def\zt{{\bold Z}_2}
\def\bs{{\bold s}}
\def\bof{{\bold f}}
\def\bq{{\bold Q}}
\def\be{{\bold e}}
\def\Hom{\on{Hom}}
\def\ker{\on{ker}}
\def\kot{\widetilde{KO}}
\def\coker{\on{coker}}
\def\da{\downarrow}
\def\colim{\operatornamewithlimits{colim}}
\def\zphat{\bz_2^\wedge}
\def\io{\iota}
\def\om{\omega}
\def\Prod{\prod}
\def\e{{\cal E}}
\def\zlt{\Z_{(2)}}
\def\exp{\on{exp}}
\def\abar{{\overline a}}
\def\xbar{{\overline x}}
\def\ybar{{\overline y}}
\def\zbar{{\overline z}}
\def\mbar{{\overline m}}
\def\nbar{{\overline n}}
\def\sbar{{\overline s}}
\def\kbar{{\overline k}}
\def\bbar{{\overline b}}
\def\et{{\widetilde E}}
\def\ni{\noindent}
\def\tsum{\textstyle \sum}
\def\coef{\on{coef}}
\def\den{\on{den}}
\def\lcm{\on{l.c.m.}}
\def\vi{v_1^{-1}}
\def\ot{\otimes}
\def\psibar{{\overline\psi}}
\def\thbar{{\overline\theta}}
\def\mhat{{\hat m}}
\def\exc{\on{exc}}
\def\ms{\medskip}
\def\ehat{{\hat e}}
\def\etao{{\eta_{\text{od}}}}
\def\etae{{\eta_{\text{ev}}}}
\def\dirlim{\operatornamewithlimits{dirlim}}
\def\gt{\widetilde{L}}
\def\lt{\widetilde{\lambda}}
\def\st{\widetilde{s}}
\def\ft{\widetilde{f}}
\def\sgd{\on{sgd}}
\def\lfl{\lfloor}
\def\rfl{\rfloor}
\def\ord{\on{ord}}
\def\gd{{\on{gd}}}
\def\rk{{{\on{rk}}_2}}
\def\nbar{{\overline{n}}}
\def\MC{\on{MC}}
\def\lg{{\on{lg}}}
\def\cH{\mathcal{H}}
\def\cS{\mathcal{S}}
\def\cP{\mathcal{P}}
\def\N{{\Bbb N}}
\def\Z{{\Bbb Z}}
\def\Q{{\Bbb Q}}
\def\R{{\Bbb R}}
\def\C{{\Bbb C}}
\def\l{\left}
\def\r{\right}
\def\mo{\on{mod}}
\def\xt{\times}
\def\notimm{\not\subseteq}
\def\Remark{\noindent{\it  Remark}}
\def\kut{\widetilde{KU}}

\def\*#1{\mathbf{#1}}
\def\0{$\*0$}
\def\1{$\*1$}
\def\22{$(\*2,\*2)$}
\def\33{$(\*3,\*3)$}
\def\ss{\smallskip}
\def\ssum{\sum\limits}
\def\dsum{\displaystyle\sum}
\def\la{\langle}
\def\ra{\rangle}
\def\on{\operatorname}
\def\od{\text{od}}
\def\ev{\text{ev}}
\def\o{\on{o}}
\def\U{\on{U}}
\def\lg{\on{lg}}
\def\a{\alpha}
\def\bz{{\Bbb Z}}
\def\eps{\varepsilon}
\def\bc{{\bold C}}
\def\bN{{\bold N}}
\def\nut{\widetilde{\nu}}
\def\tfrac{\textstyle\frac}
\def\b{\beta}
\def\G{\Gamma}
\def\g{\gamma}
\def\zt{{\Bbb Z}_2}
\def\zth{{\bold Z}_2^\wedge}
\def\bs{{\bold s}}
\def\bx{{\bold x}}
\def\bof{{\bold f}}
\def\bq{{\bold Q}}
\def\be{{\bold e}}
\def\lline{\rule{.6in}{.6pt}}
\def\xb{{\overline x}}
\def\xbar{{\overline x}}
\def\ybar{{\overline y}}
\def\zbar{{\overline z}}
\def\ebar{{\overline \be}}
\def\nbar{{\overline n}}
\def\ubar{{\overline u}}
\def\bbar{{\overline b}}
\def\et{{\widetilde e}}
\def\lf{\lfloor}
\def\rf{\rfloor}
\def\ni{\noindent}
\def\ms{\medskip}
\def\Khat{{\widehat K}}
\def\what{{\widehat w}}
\def\Yhat{{\widehat Y}}
\def\abar{{\overline{a}}}
\def\minp{\min\nolimits'}
\def\mul{\on{mul}}
\def\N{{\Bbb N}}
\def\Z{{\Bbb Z}}
\def\S{\Sigma}
\def\Q{{\Bbb Q}}
\def\R{{\Bbb R}}
\def\C{{\Bbb C}}
\def\notint{\cancel\cap}
\def\cS{\mathcal S}
\def\cR{\mathcal R}
\def\el{\ell}
\def\TC{\on{TC}}
\def\wgt{\on{wgt}}
\def\wpt{\widetilde{p_2}}
\def\dstyle{\displaystyle}
\def\Om{\Omega}
\def\ds{\dstyle}
\def\tz{tikzpicture}
\def\zcl{\on{zcl}}
\def\Vb#1{{\overline{V_{#1}}}}
\title
{On the unordered configuration space $C(RP^n,2)$}
\author{Donald M. Davis}
\address{Department of Mathematics, Lehigh University\\Bethlehem, PA 18015, USA}
\email{dmd1@lehigh.edu}
\date{May 24, 2019}


\maketitle
\begin{abstract} We prove that, if $n$ is a 2-power, the unordered configuration space $C(RP^n,2)$ cannot be immersed in $\R^{4n-2}$ nor embedded as a closed subspace of $\R^{4n-1}$, optimal results, while if $n$ is not a 2-power, $C(RP^n,2)$ can be immersed in $\R^{4n-3}$. We also obtain cohomological lower bounds for the topological complexity of $C(RP^n,2)$, which are nearly optimal when $n$ is a 2-power. We also give a new description of the mod-2 cohomology algebra of the Grassmann manifold $G_{{n+1},2}$.
 \end{abstract}
\section{Nonimmersions, nonembeddings, and immersions of $C(RP^n,2)$}
If $M$ is an $n$-manifold, the unordered configuration space of two points in $M$, $C(M,2)=(M\times M-\Delta)/\zt$, is a noncompact $2n$-manifold, and hence can be immersed in $\R^{4n-1}$ (\cite{Wh}) and embedded as a closed subspace of $\R^{4n}$.(\cite{Gan}) We prove the following optimal nonimmersion and nonembedding theorem for $C(P^n,2)$ when $n$ is a 2-power. Here $P^n$ denotes $n$-dimensional real projective space.
\begin{thm}\label{main} If $n$ is a $2$-power, $C(P^n,2)$ cannot be immersed in $\R^{4n-2}$ nor embedded as a closed subspace of $\R^{4n-1}$.\end{thm}
This will be accomplished by showing that the Stiefel-Whiney class $w_{2n-1}$ of its stable  normal bundle is nonzero. The implication for embeddings of noncompact manifolds, which is not so well-known as that for immersions, is proved in \cite[Cor 11.4]{MS}.

For contrast, we prove the following immersion theorem.
\begin{thm}\label{imm} If $n$ is not a $2$-power, then $C(P^n,2)$ can be immersed in $\R^{4n-3}$.\end{thm}

This work was motivated by a question of Mike Harrison. In \cite{Har}, he introduces the notion of totally nonparallel immersions and proves that if a manifold $M$ admits a totally nonparallel immersion in $\R^k$, then $C(M,2)$ immerses in $\R^k$. Thus we obtain a result about nonexistence of totally nonparallel immersions of 2-power real projective spaces.

\begin{proof}[Proof of Theorem \ref{main}] We denote $C_n=C(P^n,2)$, which we think of as the space of unordered pairs of distinct lines through the origin in $\R^{n+1}$. Also, $W_n$ denotes the subspace consisting of unordered pairs of orthogonal lines through the origin in $\R^{n+1}$, and $G_n$ the Grassmann manifold, usually denoted $G_{n+1,2}$, of 2-planes in $\R^{n+1}$. There is a deformation retraction $C_n\mapright{p_1} W_n$ described in \cite[p.324]{Fed}, which we will discuss thoroughly in our proof of Lemma \ref{lem}, and also an obvious map $W_n\mapright{p_2}G_n$, which is a $P^1$-bundle.

We will work only with $\zt$-cohomology. In Section \ref{cohsec}, we give a new description of the algebra $H^*(G_n)$. Here we describe just the part needed in this proof, which was first obtained by Feder in \cite[Cor 4.1]{Fed}.
 The algebra $H^*(G_n)$ is generated by classes $x=w_1$ and $y=w_2$ modulo two relations which cause the top two groups to be $H^{2n-2}(G_n)=\zt$ (resp.~$H^{2n-3}(G_n)=\zt$) with $x^{2i}y^{n-1-i}\ne0$ (resp.~$x^{2i-1}y^{n-1-i}\ne0$) iff $i=2^t-1$ for $t\ge0$ (resp.~$t\ge1$) and $2^t\le n$.  By \cite[Thm 4.3]{Fed}, $p_2^*$ is injective and \begin{equation}\label{HW}H^*(W_n)\approx H^*(G_n)[u]/(u^2=xu),\end{equation} with $|u|=1$. Also, $\sq^1y=xy$.

Let $\tau$ denote the tangent bundle,  $\eta$ a stable normal bundle, and $w$ the total Stiefel-Whitney class of a bundle. In \cite[(3)]{Opr}, it is shown that
\begin{equation}\label{Opeq}w(\tau(G_n))=(1+x)^{-2}(1+x+y)^{n+1}.\end{equation}
The map $p_2$ induces a surjective vector bundle  homomorphism $\tau(W_n)\to \tau(G_n)$, and hence a surjective  homomorphism
$$\widetilde{p_2}:\tau(W_n)\to p_2^*\tau(G_n)$$
of vector bundles over $W_n$.
Then $\ker(\wpt)$ is a line-bundle over $W_n$, and there is a vector bundle isomorphism
$$\ker(\wpt)\oplus p_2^*\tau(G_n)\approx \tau(W_n).$$
Thus
\begin{equation}\label{w}w(\tau(W_n))=(1+w_1(\ker(\wpt)))(1+x)^{-2}(1+x+y)^{n+1}.\end{equation}

By the Wu formula, $w_1(\tau(W_n))$ equals the element $v_1$ of $H^1(W_n)$ for which
$$\sq^1=\cdot v_1:H^{2n-2}(W_n)\to H^{2n-1}(W_n).$$
Since, for $j>0$, $\sq^1(x^{2^{j+1}-2}y^{n-2^j})=0$ and $$\sq^1(x^{2^{j+1}-3}y^{n-2^j}u)=x^{2^{j+1}-2}y^{n-2^j}u+nx^{2^{j+1}-2}y^{n-2^j}u+x^{2^{j+1}-3}y^{n-2^j}\cdot xu=nx^{2^{j+1}-2}y^{n-2^j}u,$$ we deduce $w_1(\tau(W_n))=nx$. From (\ref{w}), we obtain
$$nx=w_1(\ker(\wpt))+(n+1)x,$$
so $w_1(\ker(\wpt))=x$ and  (\ref{w}) becomes
$$w(\tau(W_n))=(1+x)^{-1}(1+x+y)^{n+1},$$
and hence
$$w(\eta(W_n))=(1+x)(1+x+y)^{-n-1}.$$

By Lemma \ref{lem}, we obtain
$$w(\eta(C_n))=(1+x)(1+x+y)^{-n-1}(1+x+u)^{-1}.$$
Since $x^iu^j=u^{i+j}$ for $j>0$,
$(1+x+u)^{-1}=1+\sum_{i\ge1}(x^i+u^i)=(1+x)^{-1}+u(1+u)^{-1}$ and
\begin{equation}w(\eta(C_n))=(1+x+y)^{-n-1}+u(1+u+y)^{-n-1}.\label{wc}\end{equation}

By (\ref{HW}), $H^*(C_n)\approx H^*(W_n)\approx H^*(G_n)\oplus uH^*(G_n)$, and the portion with the $u$ will always give a stronger result than the portion without.
Thus  the relevant part of $w(\eta(C_n))$ is \begin{equation}\label{alln}\ds\sum_{j,k}\tbinom{-n-1}j\tbinom{-n-1-j}ku^{k+1}y^j.\end{equation} The top dimension $H^{2n-1}(C_n)=\zt$ has as its only nonzero monomials $u^{2^t-1}y^{n-2^{t-1}}$ (all equal), and so
\begin{eqnarray*}w_{2n-1}(\eta(C_n))&=&\sum_t\tbinom{-n-1}{n-2^{t-1}}\tbinom{-2n-1+2^{t-1}}{2^t-2}\\
&=&\sum_t\tbinom{2n-2^{t-1}}{n-2^{t-1}}\tbinom{2n+2^{t-1}-2}{2^t-2}.\end{eqnarray*}
Using Lucas's Theorem, it is easy to see that $\tbinom{2n-2^{t-1}}{n-2^{t-1}}$ is odd iff $n$ is a 2-power, and when $n$ is a 2-power and $2^{t-1}\le n$, $\tbinom{2n+2^{t-1}-2}{2^t-2}$ is odd iff $t=1$, proving the theorem.

\end{proof}

The following lemma was used above
\begin{lem}\label{lem} With notation as above, $w(\tau(C_n))=(1+x+u)w(\tau(W_n))$.\end{lem}
\begin{proof} The map $p_1:C_n\to W_n$ is defined as follows. For distinct lines $\ell$ and $\ell'$, working in their plane, let $m$ and $m'$ be the pair of orthogonal lines bisecting the two angles between $\ell$ and $\ell'$, and then let $k$ and $k'$ be $45^{\text o}$ rotations of $m$ and $m'$. Then $p_1(\{\ell,\ell'\})=\{k,k'\}$, and the homotopy from the identity map of $C_n$ to $i\circ p_1$ moves $\ell$ and $\ell'$ uniformly toward the closer of $k$ and $k'$. Here $i$ is the inclusion of $W_n$ in $C_n$.
 Two scenarios for this are illustrated in Figure \ref{expl}.

\begin{fig}\label{expl}
{\bf The map $C_n\to W_n$}
\begin{center}
\begin{\tz}[scale=.75]
\draw (0,0) -- (4,0);
\draw (6,0) -- (10,0);
\draw (2,-2) -- (2,2);
\draw (8,-2) -- (8,2);
\node at (4.3,0) {$m$};
\node at (10.3,0) {$m$};
\node at (2,-2.3) {$m$};
\node at (8,-2.3) {$m$};
\draw (3.6,1.6) -- (.4,-1.6);
\draw (9.6,1.6) -- (6.4, -1.6);
\draw (3.6,-1.6) -- (.4,1.6);
\draw (9.6,-1.6) -- (6.4,1.6);
\node at (3.7,1.7) {$k$};
\node at (.3,1.7) {$k$};
\node at (9.7,1.7) {$k$};
\node at (6.3,1.7) {$k$};
\draw [ultra thick] (2.3,1.9) -- (1.7,-1.9);
\draw [ultra thick] (1.7,1.9) -- (2.3,-1.9);
\draw [ultra thick] (6.4,-.7) -- (9.6,.7);
\draw [ultra thick] (6.4,.7) -- (9.6,-.7);
\node at (1.66,2.15) {$\ell$};
\node at (2.34,2.15) {$\ell'$};
\node at (9.8,.7) {$\ell'$};
\node at (6.2,.7) {$\ell$};
\draw [->] (1.79,1.33) -- (1.04,.96);
\draw [->] (2.21,1.33) -- (2.96,.96);
\draw [->] (9.28,.56) -- (9,1);
\draw [->] (6.72,.56) -- (7,1);
\end{\tz}
\end{center}
\end{fig}

Let $Z_n$ be the space of ordered pairs of orthogonal lines in $\R^{n+1}$, and $Z_n^+$ the space of ordered pairs of orthogonal lines in $\R^{n+1}$ together with an orientation on the plane which they span. Let $Z_n^+\mapright{p} W_n$ forget the order and the orientation. This $p$ is a 4-sheeted covering space. Suppose $p$ has a section $s_\a$ on an open set $U_\a$ of $W_n$. If $p_1(\{\ell,\ell'\})=\{k,k'\}\in U_\a$, then $s_\a$ specifies an order $(k_1,k_2)$ on $\{k,k'\}$ and an orientation on the plane containing these vectors. A local trivialization of $p_1$ is defined by maps $h_\a:p_1^{-1}(U_\a)\to U_\a\times\R$ with $h_\a(\{\ell,\ell'\})=(p_1(\{\ell,\ell'\}),\tan(2\theta))$, where $\theta\in(-\frac{\pi}4,\frac{\pi}4)$ is the angle, with respect to the orientation, through which $\ell$ or $\ell'$ was rotated to end at $k_1$. Thus $p_1$ is a line bundle $\theta$ over $W_n$.

Reversing the order of $(k_1,k_2)$ in $s_\a$ negates $h_\a$, as does reversing the orientation selected by $s_\a$. Thus our line bundle $\theta$ is $L_R\otimes L_O$, where $L_R$ is the line bundle (named for Reversing) over $W_n$ associated to the double cover $Z_n\to W_n$, and $L_O$ is the line bundle (named for Orientation) over $W_n$ associated to the pullback over $W_n$ of the double cover $G_n^+\to G_n$ from the oriented Grassmannian to the unoriented one. Thus $w_1(\theta)=w_1(L_R)+w_1(L_O)$.

Clearly $w_1(L_O)$ equals $p_2^*$ of the universal $w_1$ of the Grassmannian, and this is our class $x$. That $w_1(L_R)=u$ is proved in \cite[Lemma 3.3 and Prop 3.5]{Han}. Our map $Z_n\to W_n$ is Handel's map $Z_{n+1,2}\to SZ_{n+1,2}$. Thus $w_1(\theta)=u+x$, establishing the lemma, since $w(\tau(C_n))=p_1^*(w(\tau(W_n)))\cdot p_1^*(w(\theta))$.
\end{proof}
The proof of Theorem \ref{main} showed that $w_{2n-1}(\eta(C_n))$ is nonzero iff $n$ is a 2-power.
We believe that Theorem \ref{main} gives all nonimmersion and nonembedding results for spaces $C(P^n,2)$ implied by Stiefel-Whitney classes of the normal bundle. Using our description of $H^*(G_n)$ in Section \ref{cohsec} and its implications for $H^*(C_n)$ along with (\ref{alln}), we have performed an extensive computer search for other results. Those which we found said that if $n=2^r+1$ (resp.~$2^r+2$ or $2^r+4$), then $w_{2n-5}(\eta(C_n))\ne0$ (resp.~$w_{2n-9}(\eta(C_n))\ne0$ or $w_{2n-17}(\eta(C_n))\ne0$), but the nonimmersion and nonembedding results for $C(P^n,2)$ implied by these are in the same dimension as the result for $C(P^{2^r},2)$, and so are implied by Theorem \ref{main}.

Now we prove the existence of immersions in $\R^{4n-3}$ when $n$ is not a 2-power. We continue to denote $C(P^n,2)$ as $C_n$.

\begin{proof}[Proof of Theorem \ref{imm}] We use obstruction theory to show that the map $C_n\to BO$ which classifies the stable normal bundle $\eta(C_n)$ factors through $BO(2n-3)$, which implies the immersion by the well-known theorem of Hirsch.(\cite{Hir}) The theory of modified Postnikov towers developed in \cite{GM} applies to the fibration $V_k\to BO(k)\to BO$ when $k$ is odd by \cite{Nuss}. The fiber $V_{k}$ is a union of Stiefel manifolds, and in our case, all we need is
$$\pi_i(V_{2n-3})=\begin{cases}0&i<2n-3\\ \zt&i=2n-3\\ 0&i=2n-2,\ n\text{ odd}\\ \zt&i=2n-2,\ n\text{ even.}\end{cases}$$
Since $H^{2n}(C_n)=0$, the only possible obstructions are in $H^{2n-2}(C_n;\pi_{2n-3}(V_{2n-3}))$ and $H^{2n-1}(C_n;\pi_{2n-2}(V_{2n-3}))$. The first obstruction is $w_{2n-2}(\eta(C_n))$, which is 0 when $n$ is not a 2-power by a calculation very similar to that in our proof of Theorem \ref{main}. This already implies the immersion when $n$ is odd. When $n$ is even, we argue similarly to \cite[Thm 2.3]{Monks}. The final obstruction has indeterminacy
$$H^{2n-3}(C_n)\xrightarrow{\sq^2+w_2}H^{2n-1}(C_n).$$
By (\ref{wc}), we have, for $n$ even, $w_2(\eta(C_n))=y+u^2+\binom{n+2}2x^2$. The nonzero element in $H^{2n-1}(C_n)$ is $x^{2^t-2}y^{n-2^{t-1}}u$ for an appropriate $t$. In $H^{2n-3}(C_n)$ there is a class $x^{2^t-3}y^{n-2^{t-1}}$ on which $\sq^2$ is 0, multiplication by $y$ and $x^2$ are 0, but multiplication by $u^2$ is nonzero. Therefore the final obstruction can be canceled if it is nonzero.
\end{proof}

\section{Cohomology of $G_{{n+1},2}$}\label{cohsec}
Descriptions of the cohomology ring (mod 2) of the Grassmann manifold $G_{{n+1},2}$ of $2$-planes in $\R^{n+1}$ were given initially by Chern (\cite{Chern}) and Borel (\cite{Borel}).
Here we present what we think is a new description that has been useful in our analysis. It is based on the description given by Feder in \cite{Fed}.
As in the proof of Theorem \ref{main}, we denote $G_{n+1,2}$ by $G_n$.
 In our proof of Theorem \ref{main}, we used \cite[Cor 4.1]{Fed} which stated that, with $x=w_1$ and $y=w_2$ the generators, in the top dimension, $H^{2n-2}(G_n)=\zt$, the nonzero monomials are those $x^{2i}y^{n-1-i}$ for which $i+1$ is a 2-power. Working backwards from this, we can prove the following result.

\begin{thm}\label{HG} In the ring $H^*(G_n)$, monomials $x^iy^j$ are independent if $i+2j<n$. For $\eps\in\{0,1\}$, if $2n-2k-\eps\ge n$, then $H^{2n-2k-\eps}(G_n)$ has basis $\b_1,\ldots,\b_k$, and $x^{2i-\eps}y^{n-k-i}$ equals the sum of those $\b_j$ for which $i+j$ is a $2$-power.
\end{thm}
\begin{proof} That the first relation occurs in grading $n$ is well-known (e.g., \cite[Prop 4.1]{Fed}).
The case $k=1$, $\eps=0$ is the result of \cite[Cor 4.1]{Fed} cited above. Multiplication by $x$ is an isomorphism $H^{2n-3}(G_n)\to H^{2n-2}(G_n)$ of groups  of order 2, implying the result when $k=1$ and $\eps=1$. We will prove the result by induction on $k$ when $\eps=0$. The induction when $\eps=1$ is identical.

Let $V_k=H^{2n-2k}(G_n)$, a vector space of dimension $k$ by Poincar\'e duality. Assume the result for $k$. Define
$$\phi=(\cdot y,\cdot x^2):V_{k+1}\to V_k\times V_k.$$
In $V_k\times V_k$, let
$$\g_1=(\b_1,0), \ \g_2=(\b_2,\b_1),\ldots,\g_k=(\b_k,\b_{k-1}),\ \g_{k+1}=(0,\b_k).$$
By the induction hypothesis,
$$\phi(x^{2i}y^{n-k-i-1})=\sum_{i+j\in P}\g_j,$$
where $P=\{1,2,4,\ldots\}$ denotes the set of 2-powers.

Let $W$ be the subspace of $V_k\times V_k$ spanned by $\g_1,\ldots,\g_{k+1}$. We will show that $\phi$ maps onto $W$. Then since $\dim(V_{k+1})=\dim(W)$, $\phi$ is injective. Let $\b_j=\phi^{-1}(\g_j)$. Then $\{\b_1,\ldots,\b_{k+1}\}$ is a basis for $V_{k+1}$, and
$$ x^{2i}y^{n-k-i-1}=\sum_{i+j\in P}\b_j,$$
extending the induction and completing the proof, once we establish the surjectivity of $\phi$ onto $W$.

Let $n=2m+\delta$ with $\delta\in\{0,1\}$. We first consider the case $k+1=m$. Letting $b_i=x^{2i}y^{m+\delta-i}\in V_{k+1}$ for $1\le i\le m$ (ignoring 1 or 2 monomials not required for the surjectivity), the matrix of $\phi$ with respect to the bases $\{b_1,\ldots,b_m\}$ and $\{\g_1,\ldots,\g_m\}$ is that of Lemma \ref{matlem}, and so $\phi$ is surjective. The cases of smaller values of $k$ have larger domain and smaller codomain, with $\phi$ being an extension of a quotient of the case $k+1=m$, and hence is surjective since the case $k+1=m$ was.
\end{proof}
\begin{lem}\label{matlem} Let $A_m$ denote the $m$-by-$m$ matrix over $\zt$ with $$a_{i,j}=\begin{cases}1&\text{if $i+j$ is a $2$-power}\\ 0&\text{if not.}\end{cases}$$
Then $\det(A_m)=1$.\end{lem}
\begin{proof} The proof is by induction on $m$. Let $m=2^e+\Delta$ with $0\le\Delta<2^e$. For $0\le i\le\Delta$, row $2^e+i$ contains a single 1, in column $2^e-i$. Subtract this row from other rows which have a 1 in column $2^e-i$. Then do a similar thing with columns $2^e+j$, $0\le j\le\Delta$. The result has $A_{2^e-\Delta-1}$ in the top left, and a $(2\Delta+1)$-by-$(2\Delta+1)$ matrix with 1's along the antidiagonal in the bottom right. All other elements are 0. By the induction hypothesis, this matrix has determinant 1.\end{proof}

In moderately large gradings, there is, for each $j$, a monomial $x^iy^\ell$ equal to $\b_j$. For example, in $H^{24}(G_{20})$, the following monomials equal $\b_1,\ldots,\b_8$, respectively:
$$x^{14}y^5,\ x^{12}y^6,\ x^{10}y^7,\ x^{24},\  x^{22}y,\ x^{20}y^2,\ x^{18}y^3,\ x^{16}y^4,$$
and a similar pattern holds in $H^i(G_{20})$ for $23\le i\le 38$. However, in $H^{22}(G_{20})$, $x^{14}y^4=\b_1+\b_9$, and there is no monomial which equals either $\b_1$ or $\b_9$. We can obtain $\b_1$ as $x^{22}+x^6y^8$, since $x^{22}=\b_5$ and $x^6y^8=\b_1+\b_5$.

\section{Topological complexity of $C(P^n,2)$}\label{TCsec}
The topological complexity $\TC(X)$ of a topological space $X$ is a homotopy invariant introduced by Farber in \cite{Far1} which is one less than the number of nice subsets $U_i$ into which $X\times X$ can be partitioned such that there is a continuous map $s_i:U_i\to X^I$ such that $s_i(x_0,x_1)$ is a path from $x_0$ to $x_1$. This is of interest (\cite{Far2}) for ordered (resp.~unordered) configuration spaces $F(X,n)$ (resp.~$C(X,n)$) as it measures how efficiently $n$ distinguishable (resp.~indistinguishable)  robots can be moved from one set of points in $X$ to another. The determination of $\TC(C(X,n))$ has been particularly difficult.(\cite{Sch},\cite{BR})

Farber showed (\cite{Far1}) that $\zcl(X)\le\TC(X)\le2\dim(X)$ if $X$ is a CW complex. Here $\zcl(X)$, the zero-divisor-cup-length, is the largest number of elements of $\ker(\Delta^*:\widetilde H^*(X\times X)\to\widetilde H^*(X))$ with nonzero product, where $\Delta$ is the diagonal map. The main theorem of this section determines $\zcl(C(P^n,2))$.
\begin{thm} If $0\le d<2^e$ and $r=\max\{s\in\Z:2^s\le d+\frac12\}$, then
$$\zcl(C(P^{2^e+d},2))=2^{e+2}+2^{r+1}-4$$
and $\TC(C(P^{2^e+d},2))\ge 2^{e+2}+2^{r+1}-4$.
\end{thm}
Since $C(P^n,2)$ has the homotopy type of the compact $(2n-1)$-manifold $W_n$ described in the proof of Theorem \ref{main}, $\TC(C(P^{2^e+d},2))\le 2^{e+2}+4d-2$. For $d=0,1,2,3,4$, the gap between our upper and lower bounds for $\TC(C(P^{2^e+d},2))$ is $1,4,6,10,10$, respectively.
\begin{proof} Let $n=2^e+d$ and let $C_n$, $W_n$, and $G_n$ be as in the proof of Theorem \ref{main}. We identify $H^*(C_n)$ with $H^*(W_n)$ and note that the impact of (\ref{HW}) is that $x^iu^j=x^{i+j-1}u$ if  $j>0$.

Let $\xbar=x\ot 1+1\ot x$, and define $\ybar$ and $\ubar$ similarly.
We claim that $\zcl(C_n)\ge 2^{e+2}+2^{r+1}-4$ since
\begin{equation}\xbar^{2^{e+1}-1}\,\ubar^{2^{e+1}-2}\,\ybar^{2^{r+1}-1}\ne0.\label{32}\end{equation}
To see this, we first note that the indicated product is, in bigrading $(2^{e+1}+2d-1,2^{e+1}+2^{r+2}-2d-4)$, equal to
$$\sum_{k,j}x^{2k-1}u^{2^{e+1}+2(d-j-k)}y^j\ot x^{2^{e+1}-2k}u^{2(j+k-d-1)}y^{2^{r+1}-1-j}.$$
Since the terms divisible by $u$ are independent from those not divisible by $u$, we restrict to terms whose right factor is not divisible by $u$, and obtain
\begin{equation}\label{num}\sum_jx^{2^{e+1}+2(d-j)-2}uy^j\ot x^{2^{e+1}-2(d-j+1)}y^{2^{r+1}-1-j}.\end{equation}
Terms with $j<d$ (resp.~$j>d$) have left (resp.~right) factor equal to 0 since $x^{2^{e+1}}=0$. Thus (\ref{num}) equals $x^{2^{e+1}-2}uy^d\ot x^{2^{e+1}-2}y^{2^{r+1}-d}$, which is nonzero by (\ref{HW}) and Theorem \ref{HG}.

To see that this bound for zcl cannot be improved, first note that the exponents of $\xbar$ and $\ubar$ in (\ref{32}) cannot be increased since $x^{2^{e+1}-1}=0$. If the exponent of $\ubar$ is increased by 1, the top term $x^{2^{e+1}-2}u\ot x^{2^{e+1}-2}u$ occurs with even coefficient  by symmetry.  The only hope of getting a larger nonzero product would be to increase the exponent of $\ybar$. We will use our analysis of $H^*(C_n)$ to see that this will fail to improve the zcl.

The key observation is that, with $n=2^e+d$ and $\delta\in\{0,1\}$, a nonzero monomial $x^su^\delta y^t$ in $H^*(C_n)$ with $t>d$ must have $s\le 2^e-2$. This will follow from Theorem \ref{HG} once we show that if $x^sy^t=x^{2i-\eps}y^{n-k-i}$ has $s\ge2^e-1$ and $t\ge d+1$, and $2\le 2j\le 2k$, then $2i+2j$ is not a 2-power. We have $2i+2j\ge2^e-1+\eps+2>2^e$. On the other hand, $2i+2j\le(2n-2k-2d-2)+2k=2^{e+1}-2$.

If $x^{i_1}u^{\eps_1}y^{j_1}\ot x^{i_2}u^{\eps_2}y^{j_2}$ appears in the expansion of $\xbar^a\,\ubar^b\,\ybar^c$ with maximal exponent sum, it should have $i_1=2^{e+1}-2$, $\eps_1=1$, and $j_1=d$, as we do not want to sacrifice $2^e$ $x$-exponents on both sides of the $\ot$. To have a monomial $x^{2^{e+1}-2}uy^d\ot x^{i_2}u^{\eps_2}y^{j_2}$ whose exponent sum exceeds our zcl bound would require $i_2+j_2+\eps_2>2^{e+1}-3+2^{r+1}-d$. If $j_2>d$, then $i_2\le2^e-2$, so we would need $j_2+\eps_2\ge2^e+2^{r+1}-d$ with strict inequality unless $i_2=2^e-2$. We also have $j_2\le2^e+d-1$, half the dimension of $W_n$. We would also need $\binom{d+j_2}d\equiv1$ mod 2. But this is impossible by Lemma \ref{comblem} unless $i_2=2^e-2$ and $j_2=2^e+2^{r+1}-d-1$. But then $|x^{2^{e+1}-2}uy^d\ot x^{i_2}u^{\eps_2}y^{j_2}|>2\dim(W_n)$.
The alternative is $j_2\le d$. But, since we need $\binom{d+j_2}d\equiv1$ mod 2, the largest such $j_2$ was what was used in obtaining our lower bound.
\end{proof}
\begin{lem}\label{comblem} If $2^r\le d<2^{r+1}$ and $2^{r+1}-d-1<j\le d-1$, then $\binom{d+j}d\equiv0\ (2)$.\end{lem}
\begin{proof} For $\binom{d+j}d$ to be odd, the binary expansions of $j$ and $d$ must be disjoint. Since $j\le2^{r+1}-1$, these 1's would have to be a subset of those of $2^{r+1}-1-d$.\end{proof}

 \def\line{\rule{.6in}{.6pt}}

\end{document}